\documentclass[11pt]{amsart}

\usepackage[a4paper,margin=26mm]{geometry}
\usepackage[T1]{fontenc}
\usepackage[utf8]{inputenc}
\usepackage{lmodern}
\usepackage{microtype}
\usepackage{amsmath,amssymb,amsthm}
\usepackage{booktabs}
\usepackage{enumitem}
\usepackage{tikz}
\usetikzlibrary{arrows.meta,calc}
\usepackage[colorlinks=true,linkcolor=blue!50!black,citecolor=blue!50!black,urlcolor=blue!60!black]{hyperref}

\newif\ifbw
\bwfalse
\ifbw
  \colorlet{corefill}{black!8}
  \colorlet{extfill}{black!16}
  \colorlet{softfill}{black!10}
  \colorlet{unusedfill}{white}
  \colorlet{unusedtext}{black!65}
  \colorlet{coreline}{black}
  \colorlet{mapcolor}{black}
  \colorlet{accentcolor}{black}
\else
  \colorlet{corefill}{blue!8}
  \colorlet{extfill}{orange!10}
  \colorlet{softfill}{orange!12}
  \colorlet{unusedfill}{gray!12}
  \colorlet{unusedtext}{gray!70}
  \colorlet{coreline}{blue!45!black}
  \colorlet{mapcolor}{red!75!black}
  \colorlet{accentcolor}{blue!65!black}
\fi

\numberwithin{equation}{section}
\newtheorem{theorem}{Theorem}[section]
\newtheorem{lemma}[theorem]{Lemma}
\newtheorem{proposition}[theorem]{Proposition}
\newtheorem{corollary}[theorem]{Corollary}

\theoremstyle{definition}

\newcommand{\Aut}{\operatorname{Aut}}
\newcommand{\Sym}{\operatorname{Sym}}
\newcommand{\cG}{\mathcal G}
\newcommand{\cT}{\mathcal T}
\newcommand{\cS}{\mathcal S}

\newcommand{\dom}{\operatorname{Dom}}
\newcommand{\range}{\operatorname{Range}}

\title{EPPA for Reducts of the Random Graph}

\author[S. Brenner]{Sofia Brenner}
\address{Fakultät für Mathematik und Informatik, Universität Leipzig, Germany}
\email{sofia.brenner@uni-leipzig.de}
\author[J. Hubi\v cka]{Jan Hubi\v cka}
\address{Charles University, Faculty of Mathematics and Physics\\Department of Applied Mathematics (KAM)\\
Prague, Czech Republic}
\email{hubicka@kam.mff.cuni.cz}
\author[M. Kone\v cn\'y]{Mat\v ej Kone\v cn\'y}
\address{Charles University, Faculty of Mathematics and Physics\\Department of Applied Mathematics (KAM)\\
Prague, Czech Republic, and Institut für Algebra, TU Dresden, Germany}
\email{matej@kam.mff.cuni.cz}

\thanks{J.H. was supported by project 25-15571S of the Czech Science Foundation (GA\v CR). M. K. received funding from the ERC (Grant Agreement no. 101071674, POCOCOP). Views and opinions expressed are however those of the authors only and do not necessarily reflect those of the European Union or the European Research Council Executive Agency.}
\date{}

\begin{document}
\setlength{\emergencystretch}{3em}
\begin{abstract}
We prove that finite two-graphs have coherent EPPA and that neither graphs
modulo global complementation nor two-graphs modulo global complementation have
EPPA.  This completes the classification of EPPA and coherent EPPA for the
reducts of the random graph, answering a question and confirming a conjecture
of Evans, Hubi\v{c}ka, Kone\v{c}n\'y and Ne\v{s}et\v{r}il.
\end{abstract}

\maketitle

\section{Introduction}

A class of finite structures $\mathcal C$ has the \emph{extension property for
partial automorphisms} (\emph{EPPA}), sometimes also called the
\emph{Hrushovski property}, if for every $A\in\mathcal C$ there is
$B\in\mathcal C$ containing $A$ as an induced substructure such that every
\emph{partial automorphism} of $A$ (that is, every isomorphism between induced
substructures of $A$) extends to an automorphism of $B$.  We call such $B$
an \emph{EPPA-witness} for $A$ and write $\Aut(B)$ for its automorphism
group.

Coherent EPPA asks in addition that the extensions can be chosen compatibly
with composition.  More precisely, an EPPA-witness $B$ for $A$ is
\emph{coherent} if there is a map $\Phi$ from partial automorphisms of $A$ to
$\Aut(B)$ such that $\Phi(p)$ extends $p$ and, whenever $p\colon D\to E$ and
$q\colon E\to F$ are partial automorphisms,
\[
 \Phi(q\circ p)=\Phi(q)\Phi(p).
\]
A class has coherent EPPA if every member has a coherent EPPA-witness.  See
Siniora and Solecki~\cite{Siniora} for more background on this notion,
and~\cite{hubicka2025twenty} for a recent survey of EPPA and related classification
problems.  One motivation for coherent EPPA is that if $M$ is a countable
homogeneous relational structure whose age has coherent EPPA, then
$\Aut(M)$ contains a dense locally finite subgroup~\cite[Proposition~5.1]{Siniora}.

Hrushovski proved in 1992 that the class of all finite graphs has
EPPA~\cite{hrushovski1992}.  This was one of the ingredients used by Hodges,
Hodkinson, Lascar and Shelah in their proof of the small index property for
the random graph~\cite{hodges1993b}, and it initiated a line of work studying EPPA
for concrete classes of finite structures (see, for example,
\cite{Evans3,Hubicka2018EPPA,eppatwographs}).  Here we consider this problem for reducts of the
random graph.

Thomas~\cite{Thomas1991} classified the first-order reducts of the countable
random graph $R$ up to first-order interdefinability.  Equivalently, he
classified the closed permutation groups on $V(R)$ containing $\Aut(R)$.
There are five of them.  Recall that a two-graph is a $3$-uniform
hypergraph in which every four vertices span an even number of hyperedges.
On the level of finite structures we will use the following representatives:
\begin{enumerate}[label=\textup{(\arabic*)},itemsep=1pt,topsep=3pt,leftmargin=2.2em]
\item the class $\cG$ of all finite graphs;
\item the class $\cS$ of all finite sets;
\item the class $\cT$ of all finite two-graphs;
\item the class $\cG^{\pm}$ of finite sets equipped with an equivalence
      relation on unordered pairs with at most two classes; equivalently,
      finite graphs modulo global complementation;
\item the class $\cT^{\pm}$ of finite sets equipped with an equivalence
      relation on unordered triples with at most two classes, where each
      equivalence class is a two-graph; equivalently, finite two-graphs modulo
      global complementation.
\end{enumerate}

Evans, Hubi\v{c}ka, Kone\v{c}n\'y and Ne\v{s}et\v{r}il~\cite{eppatwographs} proved
that finite two-graphs have EPPA.  They asked whether two-graphs have coherent
EPPA and, in Remark~8.2, observed that EPPA for the last two classes in the
list above was open and conjectured that neither class has EPPA.

In this paper we answer the question and confirm the conjecture.  In
particular, we obtain the complete classification.

\begin{theorem}\label{thm:classification}
For the five reducts of the random graph, EPPA and coherent EPPA are as
follows.
\[
\begin{array}{c@{\qquad}c@{\qquad}c}
\toprule
\text{finite structures} & \text{EPPA} & \text{coherent EPPA}\\
\midrule
\cG & \text{yes} & \text{yes}\\
\cS & \text{yes} & \text{yes}\\
\cT & \text{yes} & \text{yes}\\
\cG^{\pm} & \text{no} & \text{no}\\
\cT^{\pm} & \text{no} & \text{no}\\
\bottomrule
\end{array}
\]
\end{theorem}

In Section~\ref{sec:complementing-graphs} we give a simple six-vertex obstruction to
EPPA for $\cG^{\pm}$.  In
Section~\ref{sec:coherent-twographs} we adapt the method of valuation
functions~\cite{hodkinson2003,eppatwographs} to prove coherent EPPA for two-graphs.  In
Section~\ref{sec:complementing-two-graphs} we show by a rooting argument that
EPPA for $\cT^{\pm}$ would imply EPPA for $\cG^{\pm}$. Finally, in Section~\ref{sec:conclusion},
we discuss EPPA for reducts of random hypergraphs, and EPPA for graphs with
loops modulo complementation.

\section{Graphs modulo global complementation}\label{sec:complementing-graphs}

We represent a structure in $\cG^{\pm}$ by choosing a graph representative
$A=(V,E)$ and writing
\[
 c_A\colon\binom V2\to\mathbb Z_2
\]
for its edge indicator; we abbreviate $\{x,y\}$ to $xy$ and call the
values of $c_A$ colours.  A partial
bijection $p\colon D\to E$ is a \emph{partial automorphism}
of the reduct if and only if it preserves equality of colours.
Equivalently, if $|D|\ge2$, there is
$\varepsilon\in\mathbb Z_2$ such that for every pair $xy\in\binom D2$ we have
\begin{equation}
 c_A(p(x)p(y))=c_A(xy)+\varepsilon.
\label{eq:pair-sign}
\end{equation}
If $|D|<2$ there is no condition.  We call
$p$ \emph{complementing} when $\varepsilon=1$.

\medskip

In the remainder of this section we will prove the following proposition:
\begin{proposition}\label{prop:a6}
Let $A_6$ be the graph shown in Figure~\ref{fig:A6} with vertex set $\{1,2,3,4,5,6\}$ and
\[
 E(A_6)=\{12,13,15,26,34,35,46\}.
\]
Put $S=\{1,2,3,4\}$ and define complementing partial automorphisms $p_5$ and $p_6$ of $A_6$ with domains $S\cup\{5\}$ and $S\cup\{6\}$ as
\[
 p_5=(1\,2\,3\,4)(5),
 \qquad
 p_6=(1\,2\,3\,4)(6).
\]
There is no finite graph modulo global complementation containing $A_6$ as an induced substructure in which each of $p_5$ and $p_6$ extends to an automorphism.
\end{proposition}

The cycle $(1\,2\,3\,4)$ reverses the pair colours on $S$, and the neighbourhoods
$\{1,3\}$ and $\{2,4\}$ of $5$ and $6$ in $S$ alternate accordingly.
Thus $p_5$ and $p_6$ are indeed complementing partial automorphisms.

Clearly, this implies the following theorem:
\begin{theorem}\label{thm:complementing-graph}
The class $\cG^{\pm}$ of finite graphs modulo global complementation does not
have EPPA.
\end{theorem}

\begin{figure}[ht]
\centering
\begin{minipage}{0.43\textwidth}
\centering
\begin{tikzpicture}[scale=.9,
  v/.style={circle,draw,inner sep=1.5pt,minimum size=18pt,font=\small},
  core/.style={v,fill=corefill}, ext/.style={v,fill=extfill},
  corebox/.style={rounded corners,draw=coreline,dashed}]
\node[core] (1) at (0,2) {$1$};
\node[core] (2) at (2,2) {$2$};
\node[core] (3) at (2,0) {$3$};
\node[core] (4) at (0,0) {$4$};
\node[ext]  (5) at (-1.6,1) {$5$};
\node[ext]  (6) at (3.6,1) {$6$};
\draw[corebox] (-.42,-.42) rectangle (2.42,2.42);
\draw (1)--(2); \draw (1)--(3); \draw (1)--(5);
\draw (2)--(6); \draw (3)--(4); \draw (3)--(5); \draw (4)--(6);
\end{tikzpicture}
\par\smallskip\small $A_6$; the boxed vertices form $S$.
\end{minipage}
\hfill
\begin{minipage}{0.52\textwidth}
\centering
\begin{tikzpicture}[scale=.78,
  v/.style={circle,draw,inner sep=1.2pt,minimum size=16pt,font=\small},
  arr/.style={-{Latex[length=2mm]},line width=.85pt,mapcolor}]
\node[v,fill=softfill] (5) at (-1.25,1) {$5$};
\node[v,fill=corefill] (1) at (0,2) {$1$};
\node[v,fill=corefill] (2) at (2,2) {$2$};
\node[v,fill=corefill] (3) at (2,0) {$3$};
\node[v,fill=corefill] (4) at (0,0) {$4$};
\draw[arr] (1)--(2); \draw[arr] (2)--(3); \draw[arr] (3)--(4); \draw[arr] (4)--(1);
\draw[arr] (5) to[out=130,in=230,looseness=7] (5);
\node at (.4,-.65) {$p_5$};
\begin{scope}[xshift=3.75cm]
\node[v,fill=corefill] (1b) at (0,2) {$1$};
\node[v,fill=corefill] (2b) at (2,2) {$2$};
\node[v,fill=corefill] (3b) at (2,0) {$3$};
\node[v,fill=corefill] (4b) at (0,0) {$4$};
\node[v,fill=softfill] (6b) at (2.9,1) {$6$};
\draw[arr] (1b)--(2b); \draw[arr] (2b)--(3b); \draw[arr] (3b)--(4b); \draw[arr] (4b)--(1b);
\draw[arr] (6b) to[out=50,in=-50,looseness=7] (6b);
\node at (1.25,-.65) {$p_6$};
\end{scope}
\end{tikzpicture}
\par\smallskip\small The two complementing partial automorphisms.
\end{minipage}
\caption{The six-vertex obstruction.  In particular $15$ is an edge and
$16$ is a non-edge.}
\label{fig:A6}
\end{figure}
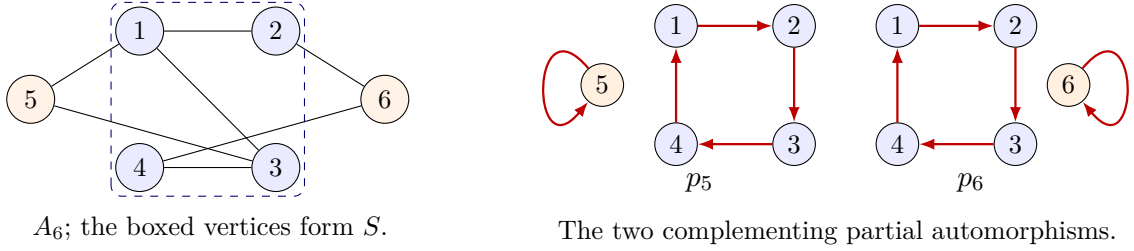

\medskip

We will need the following simple lemma. 
We use left actions, so $(gh)x=g(hx)$.  For a set $\Omega$, write
$\Sym(\Omega)$ for its symmetric group.  If $L\le\Sym(\Omega)$ and
$x\in\Omega$, put $Lx=\{\ell x:\ell\in L\}$.  We say that $g$
\emph{normalises} $L$ if $gLg^{-1}=L$.

\begin{lemma}\label{lem:coset}
Let $\Omega$ be finite, let $b$ be a nonnegative integer, let
$L\le\Sym(\Omega)$, and let $g\in\Sym(\Omega)$ normalise $L$.
If every element of $Lg$ fixes at
most $b$ points of $\Omega$, then all points fixed by elements of $Lg$
lie in at most $b$ orbits of $L$.
\end{lemma}

\begin{proof}
Let
\[
 I=\{(\ell,z)\in L\times\Omega:(\ell g)z=z\}.
\]
Each $\ell\in L$ contributes at most $b$ incidences, so $|I|\le b|L|$.

Suppose that an $L$-orbit $O$ contains a point $z$ fixed by some element
$\ell g\in Lg$.  Then $gz=\ell^{-1}z\in O$.
Since $g$ normalises $L$, it permutes the $L$-orbits, and consequently
$gO=O$.

For any $g$-invariant $L$-orbit $O$ and any $z\in O$, there is
$\ell_z\in L$ with $\ell_z(gz)=z$.  The set of all
$\ell\in L$ satisfying $\ell(gz)=z$ is precisely $L_z\ell_z$,
where $L_z$ is the stabiliser of $z$ in $L$; it therefore has size
$|L_z|$.  Fix $z_0\in O$.  Since $L$ acts transitively on $O$, every
$L_z$ ($z\in O$) has size $\frac{|L|}{|O|}$.  Counting incidences by their
second coordinate, we obtain
\[
\begin{aligned}
 |I\cap(L\times O)|
 &=\sum_{z\in O}|\{\ell\in L:(\ell g)z=z\}|\\
 &=\sum_{z\in O}|L_z|=|O|\,|L_{z_0}|=|L|.
\end{aligned}
\]
Since the incidences coming from different $L$-orbits are disjoint and
$|I|\le b|L|$, at most $b$ such orbits can occur.
\end{proof}

\begin{proof}[Proof of Proposition~\ref{prop:a6}]
Suppose $B\in\cG^{\pm}$ contains $A_6$ and both $p_5,p_6$ extend
to automorphisms $g_5,g_6\in\Aut(B)$.  Since $A_6$ realises both
equivalence classes on unordered pairs, so does $B$. 
The extensions $g_5,g_6$ are necessarily complementing.

Put $L=\Aut(B)_{(S)}$, the pointwise stabiliser of $S$ in $\Aut(B)$. Since
every complementing automorphism fixes at most one vertex (two fixed vertices
would give an invariant unordered pair whose colour would have to be swapped),
no element of $L$ is complementing.

Since $g_5$ permutes $S$ setwise, it normalises $L$.  As
$g_5$ and $g_6$ agree on $S$, we have that
\[
 d=g_6g_5^{-1}\in L,
 \qquad g_6=dg_5\in Lg_5.
\]
Every element of $Lg_5$ is complementing and fixes at most one point, but $g_5$
fixes $5$ and $g_6$ fixes $6$.  Lemma~\ref{lem:coset} gives $L5=L6$.
Choose $\ell\in L$ with $\ell5=6$.  Then $\ell$ is not complementing and fixes $1$, so
\[
 c_B(15)=c_B((\ell1)(\ell5))=c_B(16),
\]
a contradiction, since $15$ is an edge of $A_6$ while $16$ is not.
\end{proof}

\section{Coherent EPPA for two-graphs}\label{sec:coherent-twographs}

We now turn to two-graphs.  A \emph{two-graph} is a 3-uniform hypergraph $(V,T)$ 
such that every four-element subset of $V$ contains an even number of hyperedges.
Equivalently, for a graph $G$ on $V$, let $T(G)$ consist of the triples
with an odd number of edges; then $(V,T(G))$ is a two-graph.  In fact, every
two-graph is of this form, and $T(G)$ depends only on the
\emph{Seidel-switching class} of $G$.

Our proof uses the method of \emph{valuation functions}, introduced by
Hodkinson and Otto~\cite{hodkinson2003} and developed further in
\cite{Evans3,Hubicka2018metricEPPA,Hubicka2018EPPA,eppatwographs,HubickaSemigeneric}; see also~\cite{hubicka2025twenty}. To recall the construction in~\cite{eppatwographs}, let $X$ be a finite nonempty set and consider the \emph{antipodal metric space of diameter 3} on
$V_X=X\times\mathbb Z_2^X$ defined below.  We call $x$ the \emph{projection} of $(x,\chi)$ and
$\chi\colon X\to\mathbb Z_2$ its \emph{valuation function}.  The metric
$d\colon V_X^2\to\{0,1,2,3\}$ is defined as follows (where $\mathbf1$ denotes the constant-one function):
\[
d((x,\chi), (y,\xi)) = \begin{cases}
  0 &\text{ if } (x,\chi) = (y,\xi)\\
  3 &\text{ if } (x,\chi) = (y,\mathbf{1}-\xi)\\
  1 &\text{ if } (x,\chi) \notin \{(y,\xi),(y,\mathbf{1}-\xi)\}\text{ and }\chi(y)+\xi(x) = 0\\
  2 &\text{ if } (x,\chi) \notin \{(y,\xi),(y,\mathbf{1}-\xi)\}\text{ and }\chi(y)+\xi(x) = 1.
\end{cases}
\]
The unordered pairs at distance~3 form the vertices of a two-graph:
three such antipodal pairs form a hyperedge if and only if the distance-1
graph on their six endpoints is a $6$-cycle.  In~\cite{eppatwographs}, a two-graph
is first represented as an antipodal metric space of diameter~3, an
EPPA-witness is constructed there, and its antipodal pairs are then
identified to recover a two-graph.

We make two adjustments to the construction in~\cite{eppatwographs}:
\begin{enumerate}
  \item Instead of working with antipodal metric spaces of diameter 3, we quotient already on the level of valuation functions. This makes it possible to choose extensions such that we obtain a homomorphism from the automorphism group of the input two-graph to the automorphism group of its EPPA-witness.
  \item Here $X$ is chosen much larger: its size is divisible by the size of every induced substructure of the input.  This allows us to extend the automorphism group of each such substructure homomorphically by repeating its action on blocks, and then obtain coherence from the criterion below.
\end{enumerate}

The following lemma is a convenient form of \cite[Lemma~4.17]{BradleyEPPAnumbers}.

\begin{lemma}\label{lem:coherentisation}
Let $A\subseteq B$ be finite structures.  Assume
\begin{enumerate}
\item every partial automorphism of $A$ extends to an automorphism of $B$;
\item for every induced $D\subseteq A$ there is a homomorphism
\[
 \rho_D\colon\Aut(D)\longrightarrow\Aut(B)
\]
extending the natural action on $D$.
\end{enumerate}
Then $B$ is a coherent EPPA-witness for $A$.
\end{lemma}

\begin{proof}
For each isomorphism class of induced substructures of $A$ choose a
representative $D_0\subseteq A$.  For every induced copy $E\subseteq A$ in
that class choose an isomorphism $u_E\colon D_0\to E$ and an automorphism
$\beta_E\in\Aut(B)$ extending $u_E$.  For a partial automorphism
$p\colon E\to F$ put
\[
 \Phi(p)=\beta_F\rho_{D_0}(u_F^{-1}pu_E)\beta_E^{-1}.
\]
Then $\Phi(p)\in\Aut(B)$ extends $p$.  If $p\colon E\to F$ and $q\colon F\to G$,
direct cancellation and
the homomorphism property of $\rho_{D_0}$ give
$\Phi(q\circ p)=\Phi(q)\Phi(p)$.
\end{proof}

\subsection*{Witness construction}
Given a finite nonempty set $X$, we now define a very symmetric two-graph $C_X$, together with a large group of automorphisms.  Later
we use the two-graphs $C_X$ as EPPA-witnesses.  Put
\[
 Q_X=\mathbb Z_2^X/\langle\mathbf1\rangle,
 \qquad V(C_X)=X\times Q_X.
\]
Here $\mathbf1$ denotes the constant-one function on $X$, so $Q_X$ consists
of functions $\chi\colon X\to\mathbb Z_2$ modulo the equivalence
$\chi\sim\chi+\mathbf1$.  We write $[\chi]$ for the equivalence class of
$\chi$, and write vertices of $C_X$ as $(x,[\chi])$. We call the first entry the \emph{projection}
and the second the \emph{valuation}, and we denote $\pi((x,[\chi]))=x$. For three distinct vertices
$v_i=(x_i,[\chi_i])$ ($i=1,2,3$), declare $\{v_1,v_2,v_3\}$ to be a triple precisely when
\begin{equation}
 t(v_1,v_2,v_3)
 =\sum_{1\le i<j\le3}
   \bigl(\chi_i(x_j)+\chi_j(x_i)\bigr)=1.
\label{eq:triple-valuation}
\end{equation}
Each pair of vertices contributes one bit to \eqref{eq:triple-valuation}.  Changing one
valuation representative by $\mathbf1$ flips the two pair-contributions
involving that vertex, so $t$ is well-defined.  On four vertices every
pair-contribution occurs twice, hence $C_X$ is a two-graph.

\subsection*{Automorphisms of \texorpdfstring{$C_X$}{C\_X}}
Next, we turn to the description of automorphisms of $C_X$ which we will later use to extend partial automorphisms.
Let $\mathcal S_X$ be the vector space of symmetric functions
$F\colon X^2\to\mathbb Z_2$, let $J_X$ be the constant-one function, and set
\[
 \overline{\mathcal S}_X=\mathcal S_X/\langle J_X\rangle.
\]
For $\sigma\in\Sym(X)$ and $F\in\mathcal S_X$ define
\begin{equation}
 \theta_{\sigma,F}(x,[\chi])=(\sigma x,[\xi]),
 \text{ where }
 \xi(\sigma y)=\chi(y)+F(x,y).
\label{eq:theta-action}
\end{equation}
The map is well-defined and bijective.  Under this map, the contribution
of any pair of vertices with first coordinates $x,y$ to \eqref{eq:triple-valuation} changes by
$F(x,y)+F(y,x)=0$, by symmetry of $F$.  Hence it is an automorphism of
$C_X$.  Moreover
\begin{equation}
 \theta_{\tau,G}\theta_{\sigma,F}
 =\theta_{\tau\sigma,F+\sigma^*G},
 \text{ where }
 (\sigma^*G)(x,y)=G(\sigma x,\sigma y),
\label{eq:theta-composition}
\end{equation}
and $\theta_{\sigma,F}=\theta_{\sigma,F+J_X}$.  Since
$\sigma^*J_X=J_X$, the definition
\[
 \sigma^*\bar F:=\overline{\sigma^*F}
\]
does not depend on the representative $F$ of $\bar F\in
\overline{\mathcal S}_X$.

Define
\[
 \Gamma_X=\Sym(X)\times\overline{\mathcal S}_X
\]
with multiplication
\[
 (\tau,\bar G)(\sigma,\bar F)
 =\bigl(\tau\sigma,\overline{F+\sigma^*G}\bigr).
\]
This makes $\Gamma_X$ a group, and \eqref{eq:theta-composition} gives a homomorphism
$\Gamma_X\to\Aut(C_X)$.  For $g\in\Gamma_X$ write $\theta_g$ for the
corresponding automorphism.

\subsection*{Extending partial automorphisms}
A subset of $C_X$ is \emph{transversal} if it contains at most one vertex
with each first coordinate.

\begin{lemma}\label{lem:transversal}
Let $\varphi$ be a partial automorphism of $C_X$ whose domain and range are transversal.
Put $D=\pi[\dom(\varphi)]$ and $V=\range(\varphi)$, and write
$U=\dom(\varphi)=\{(x,[\chi_x]):x\in D\}$.  Since $V$ is transversal, $\varphi$ induces an injection
$\sigma_0\colon D\to X$ on the first coordinates.  For $x\in D$ write
$\varphi(x,[\chi_x])=(\sigma_0x,[\eta_{\sigma_0x}])$;
this defines the valuation classes $[\eta_z]$ for $z\in\sigma_0[D]$.

Then for every permutation $\sigma\in\Sym(X)$ extending
$\sigma_0$ there is $\bar F\in \overline{\mathcal S}_X$ such that
$\theta_g$ extends $\varphi$, where $g=(\sigma,\bar F)\in\Gamma_X$.
Moreover, if $D=X$ then $g$ is unique.
\end{lemma}

\begin{proof}
Choose representatives $\chi_x$ and $\eta_{\sigma_0x}$ of the valuation
classes appearing in $U$ and $V$.  For distinct $x,y\in D$ put
\[
 \delta(x,y)=
 \chi_x(y)+\chi_y(x)
 +\eta_{\sigma x}(\sigma y)+\eta_{\sigma y}(\sigma x).
\]
Since $\varphi$ preserves the two-graph relation, the values of $\delta$ satisfy the \emph{triple identity}
$\delta(x,y)+\delta(x,z)+\delta(y,z)=0$ for all distinct $x,y,z\in D$.
We claim that there is $c\colon D\to\mathbb Z_2$ such that, for all distinct $x,y\in D$,
\begin{equation}
 \delta(x,y)=c(x)+c(y).
\label{eq:cut}
\end{equation}
(This is the usual cut arising in Seidel switching.)
If $D=\varnothing$, use $c=\varnothing$.  Otherwise choose $x_0\in D$, put $c(x_0)=0$ and
$c(x)=\delta(x_0,x)$ for $x\ne x_0$.
The pairs containing $x_0$ satisfy \eqref{eq:cut} by definition, and
all other pairs by the triple identity for $x_0,x,y$. 

For $x\in D$ and $y\in X$, put
\begin{equation}
 F(x,y)=F(y,x)=\chi_x(y)+\eta_{\sigma x}(\sigma y)+c(x).
\label{eq:matrix-row}
\end{equation}
Note that this is well-defined even if $y\in D$: this follows from~\eqref{eq:cut}.
If $x,y\in X\setminus D$ put $F(x,y)=0$.
For $x\in D$, the valuation of
$\theta_{\sigma,F}(x,[\chi_x])$ differs from $\eta_{\sigma_0x}$ by the
constant function $c(x)$, so $\theta_{\sigma,F}$ extends $\varphi$.

Assume now $D=X$.  If $F'$ gives another extension with the same $\sigma$,
then for each $x$ the row $F'(x,-)+F(x,-)$ is constant, say $c(x)$.
Symmetry forces all $c(x)$ to be equal (the one-point case is immediate), so
$F'=F$ or $F'=F+J_X$.  Thus $\bar F$, and hence $g$, is unique.
\end{proof}

The following lemma is key for coherence.

\begin{lemma}\label{lem:group-extension}
Let $D\subseteq X$ be nonempty, with $|D|$ dividing $|X|$, and let
\[
 V_D=\{(d,[\chi_d]):d\in D\}\subseteq C_X
\]
be transversal.  There is a homomorphism
\[
 \rho_D\colon\Aut(V_D)\longrightarrow\Aut(C_X)
\]
extending the natural action on $V_D$.
\end{lemma}

\begin{proof}
Restrict the valuations to $D$:
\[
 U_D=\{(d,[\chi_d|_D]):d\in D\}\subseteq C_D.
\]
Note that a triple of vertices from $V_D$ forms a hyperedge in $C_X$ if and only
if their corresponding restrictions in $U_D$ form a hyperedge in $C_D$.
Consequently, automorphisms of $U_D$ and $V_D$ induce the same permutation
group on $D$.  We identify these two groups via their actions on the first
coordinates.
By the uniqueness part of Lemma~\ref{lem:transversal}, each
$h\in\Aut(U_D)$ has a unique lift $s_D(h)\in\Gamma_D$, and uniqueness implies
that $s_D$ is a homomorphism.

Partition $X$ into blocks of size $|D|$, with $D$ itself as one block.
For every block fix a bijection to $D$, taking the identity on $D$, and let
$r\colon X\to D$ be the resulting retraction.  Inflate $\sigma\in\Sym(D)$
to the permutation $\widehat\sigma$ acting as $\sigma$ on each block via
these bijections, so
\begin{equation}
 r(\widehat\sigma x)=\sigma(r(x)).
\label{eq:retraction-equivariance}
\end{equation}
Pull back symmetric matrices by
\[
 (r^*F)(x,y)=F(r(x),r(y)).
\]
Since $r^*J_D=J_X$, this is well-defined on quotient classes.  Equation
\eqref{eq:retraction-equivariance} gives
\begin{equation}
 r^*(\sigma^*\bar F)=\widehat\sigma^{\,*}(r^*\bar F).
\label{eq:matrix-pullback}
\end{equation}
By \eqref{eq:matrix-pullback}, the map $\iota_r\colon\Gamma_D\longrightarrow\Gamma_X$ defined by
$\iota_r(\sigma,\bar F)=(\widehat\sigma,r^*\bar F)$
is a group homomorphism.

There is also an embedding $j_r\colon C_D\longrightarrow C_X$ given by $j_r(d,[\chi])=(d,[\chi\circ r])$. 
A direct substitution in \eqref{eq:theta-action}, using \eqref{eq:retraction-equivariance}, gives $j_r(\theta_g(v))=\theta_{\iota_r(g)}(j_r(v))$
for every $g\in\Gamma_D$ and $v\in C_D$.
In particular, for $h\in\Aut(U_D)$, the automorphism
$\theta_{\iota_r(s_D(h))}$ of $C_X$ restricts on $W_D=j_r(U_D)$ to the
action of $h$ transported by $j_r$.

Finally, the map from $W_D$ to $V_D$ which preserves the first
coordinate is an isomorphism: the corresponding valuations agree when
restricted to $D$, since $r|_D=\mathrm{id}_D$.  Lemma~\ref{lem:transversal},
with first-coordinate permutation $\mathrm{id}_X$, extends this map to some
$\gamma\in\Aut(C_X)$.  Therefore
\[
 \rho_D(h)=
 \gamma\,\theta_{\iota_r(s_D(h))}\,\gamma^{-1}
\]
is the desired homomorphism.
\end{proof}

\bigskip

We are now ready to prove coherent EPPA for two-graphs.

\begin{theorem}\label{thm:coherent-twographs}
The class $\cT$ of all finite two-graphs has coherent EPPA.
\end{theorem}

\begin{proof}
Let $A = (V(A), T(A))$ be a two-graph with $n\ge1$
vertices and put
\[
 N=\operatorname{lcm}(1,2,\dots,n).
\]
Choose a set $X\supseteq V(A)$ of size $N$.

Choose a graph $G$ on $V(A)$ with $T(G)=T(A)$ (a representative of the
switching class of $A$), and a linear order on $X$.  For each $x\in V(A)$
choose a valuation $\chi_x\colon X\to\mathbb Z_2$ such that for $x<y$ in $V(A)$,
\[
 \chi_x(y)=0,
 \qquad
 \chi_y(x)=\mathbf1_{xy\in E(G)},
\]
and fill unused values arbitrarily.  By \eqref{eq:triple-valuation}, the map
$x\mapsto(x,[\chi_x])$ is an embedding of $A$ to a transversal copy in $C_X$.

Every partial automorphism of this copy induces a partial bijection of the
first coordinates, which extends to a permutation of $X$; hence by
Lemma~\ref{lem:transversal}, $C_X$ is an EPPA-witness for $A$.  For every
nonempty induced $D\subseteq A$, the number $|D|$ divides $N$, so
Lemma~\ref{lem:group-extension} gives a homomorphic extension of
$\Aut(D)$ in $\Aut(C_X)$; for empty $D$, use the trivial homomorphism.  Now apply
Lemma~\ref{lem:coherentisation}.
\end{proof}

\section{Two-graphs modulo global complementation}\label{sec:complementing-two-graphs}

We finally consider $\cT^{\pm}$.  Recall that a structure in this class is
a finite set equipped with an equivalence relation on unordered triples with at
most two classes, each class being a two-graph.  Choose a representative
two-graph and write $t$ for its triple indicator.  A partial bijection
$p\colon D\to E$ is a partial automorphism if and only if, whenever
$|D|\ge3$, there is a single $\varepsilon\in\mathbb Z_2$ such that
\[
 t(p(X))=t(X)+\varepsilon
\]
for every $X\in\binom D3$; for $|D|<3$ there is no condition.

The negative result now follows from the graph case by rooting the
two-graph:  Once a vertex $*$ is distinguished, the triples containing $*$
determine a graph on the remaining vertices, and the four-set parity
condition determines all other triples.

\begin{proposition}\label{prop:rooting}
If $\cT^{\pm}$ has EPPA, then $\cG^{\pm}$ has EPPA.
\end{proposition}

\begin{proof}
Let $A$ be a graph modulo complementation, choose a representative with edge
indicator $c$, and add a new vertex $*$.  Define a two-graph on
$V(A)\cup\{*\}$ by taking the two-graph associated with the graph in which
$*$ is isolated.  Explicitly,
\begin{equation}
 t(*,x,y)=c(xy),
 \qquad
 t(x,y,z)=c(xy)+c(xz)+c(yz).
\label{eq:rooting-triples}
\end{equation}
Regard this as an object $\widehat A\in\cT^{\pm}$.

Let $p\colon D\to E$ be a partial automorphism of $A$.  If $|D|\ge2$, by
\eqref{eq:pair-sign} there is $\varepsilon\in\mathbb Z_2$ such that
\[
 c(p(x)p(y))=c(xy)+\varepsilon
\]
for all pairs in $D$; for $|D|\le1$ there is nothing to check.
Triples containing $*$ change colour by $\varepsilon$, while triples
disjoint from $*$ change colour by $3\varepsilon=\varepsilon$.  Thus
$\widehat p=p\cup\{*\mapsto *\}$ is a partial automorphism of
$\widehat A$ modulo complementation.

Assume that $C\supseteq\widehat A$ is an EPPA-witness in $\cT^{\pm}$.
Choose a representative two-graph $T_C$ whose restriction to $\widehat A$
is the representative in \eqref{eq:rooting-triples}.  If $C$ realises two triple classes, choose
the appropriate one of them as the triple relation.  If $C$ has at most one
triple class, then the restriction \eqref{eq:rooting-triples} is constant (or
contains no triples), and we take $T_C$ to be the empty or the complete
two-graph so that the restriction agrees.
Define a graph $B$ on $V(C)\setminus\{*\}$ by
\[
 c_B(xy)=\mathbf1_{\{*,x,y\}\in T_C}.
\]
Then $B$ extends the chosen representative of $A$.

For a partial automorphism $p$ of $A$, extend $\widehat p$ to
$g\in\Aut(C)$.  Since $g$ fixes $*$, it permutes
$V(C)\setminus\{*\}$.  If $C$ realises two triple classes, $g$ either
preserves or swaps them; in the at-most-one-class case the chosen
representative is constant and $g$ preserves it.  Thus in either case there is
$\delta_g\in\mathbb Z_2$ such that
\[
 c_B(gx,gy)
 =t_C(*,gx,gy)
 =t_C(*,x,y)+\delta_g
 =c_B(xy)+\delta_g.
\]
Thus the restriction of $g$ is an automorphism of $B$ modulo global
complementation extending $p$.  Hence $B$ is an EPPA-witness for $A$ in
$\cG^{\pm}$.
\end{proof}

\begin{corollary}\label{cor:complementing-two-graphs}
The class $\cT^{\pm}$ does not have EPPA.
\end{corollary}
\begin{proof}
By Proposition~\ref{prop:rooting}, EPPA for $\cT^{\pm}$ would imply EPPA
for $\cG^{\pm}$, contradicting Theorem~\ref{thm:complementing-graph}.
\end{proof}

We can now prove Theorem~\ref{thm:classification}:
\begin{proof}[Proof of Theorem~\ref{thm:classification}]
The free-amalgamation classes $\cG$ and $\cS$ have coherent EPPA by
Siniora and Solecki~\cite[Corollary~4.6]{Siniora}.
Theorem~\ref{thm:coherent-twographs} handles $\cT$, while
Theorem~\ref{thm:complementing-graph} and
Corollary~\ref{cor:complementing-two-graphs} handle the remaining two
classes.  Coherent EPPA implies EPPA.
\end{proof}

\section{Conclusion}\label{sec:conclusion}

Although EPPA and coherent EPPA coincide for the five reducts considered here,
they need not coincide in general.  Evans asked whether EPPA for the age of a
homogeneous $\omega$-categorical structure implies that its automorphism group
contains a dense locally finite subgroup (see~\cite[Question~3.1]{Bhattacharjee2005}
and~\cite{Sagi2024}).  Siniora subsequently asked whether EPPA implies coherent
EPPA and whether the existence of a dense locally finite subgroup implies
coherent EPPA~\cite[Questions~3 and~4]{Siniora2}.  In forthcoming work on
generalised semigeneric structures, examples separating all three properties
are constructed, giving negative answers to these questions.

\subsection{Graphs with loops}\label{sec:loops}

Curiously, EPPA holds for a variant of $\cG^{\pm}$. Concretely,
Beliayeu proved in his bachelor's thesis~\cite{Beliayeu2023bc} that if
loops are allowed and global complementation is required to complement the
loops together with all other edges, then EPPA holds.  Here a graph with loops is a finite set with a symmetric binary
relation $E$, with loops allowed; a partial bijection is a
\emph{morphism} if it either preserves or complements $E$ on its
domain, including the loops.  Unlike in the loopless case, a
complementing automorphism cannot fix even one vertex, since it would
have to complement the loop at that vertex. 

\begin{proposition}[Beliayeu~\cite{Beliayeu2023bc}]\label{prop:beliayeu}
Every finite graph with loops has a finite extension in which every partial
morphism extends to a total morphism of the same kind.
\end{proposition}

\begin{proof}[Proof sketch]
Let $A=(V,E)$ and fix a linear order on $V$.  Let $P$ be the graph on
$\{0,1,2,3\}$ consisting of the path $0-1-2-3$ together with loops at $1$ and
$2$.  The permutation
\[
 \sigma_P=(0\,1\,3\,2)
\]
is complementing on $P$, while
$\alpha_P=\sigma_P^2=(0\,3)(1\,2)$ is an automorphism.

Let $B$ have vertices
\[
 (x,a,\lambda),\qquad x\in V,\quad a\in P,\quad
 \lambda\colon V\setminus\{x\}\to\mathbb Z_2,
\]
and put, for $u=(x,a,\lambda)$ and $v=(y,b,\mu)$,
\begin{equation}
 uEv
 \quad\Longleftrightarrow\quad
 \begin{cases}
   aE_Pb,&x=y,\\
   \lambda(y)\ne\mu(x),&x\ne y.
 \end{cases}
\label{eq:looped-adjacency}
\end{equation}
Embed $A$ by
\[
 x\longmapsto (x,a_x,\lambda_x),\qquad
 a_x=\mathbf1_{xEx},\qquad
 \lambda_x(y)=\mathbf1_{\{y<x\text{ and }xEy\}}.
\]

We will need four kinds of morphisms of $B$.  Define
\[
 s(x,a,\lambda)=(x,\sigma_P(a),\lambda'), \text{ where }
 \lambda'(y)=\lambda(y)+\mathbf1_{x<y};
\]
this is complementing.  Every $\pi\in\Sym(V)$ induces an automorphism
$\widehat\pi(x,a,\lambda)=(\pi x,a,\lambda^\pi)$, where
$\lambda^\pi(\pi y)=\lambda(y)$.  For $i\in V$, let $\beta_i$ send $(i,a,\lambda)$ to
$(i,\alpha_P(a),\lambda)$ and fix all vertices whose first coordinate
is different from $i$.  For distinct $i,j\in V$, let $\gamma_{ij}$
change $\lambda(j)$ to $\lambda(j)+1$ in every vertex with first
coordinate $i$, and change $\lambda(i)$ to $\lambda(i)+1$ in every
vertex with first coordinate $j$, leaving all other entries unchanged.
The functions $\beta_i$ and $\gamma_{ij}$ are automorphisms, and the maps within
each family commute.  These facts follow directly from \eqref{eq:looped-adjacency}.

Let $p$ be a partial morphism of the embedded copy of $A$, and let
$\varepsilon\in\mathbb Z_2$ record whether it preserves or complements $E$.  Extend
the induced partial permutation of the first coordinates to
$\pi\in\Sym(V)$ and start with
\[
 h_0=\widehat\pi\circ s^\varepsilon.
\]
The map $h_0$ is complementing if and only if $p$ is,
and agrees with $p$ on the first coordinates.  Since $p$ and $h_0$ have
the same effect on loops, suitable $\beta_i$ correct the second
coordinates, giving a morphism $h_1$ which agrees with $p$ in its first
two coordinates.

It remains to correct the valuations.  Let $R\subseteq V$ consist of
the first coordinates of the vertices in the range of $p$.  Since $h_1$
and $p$ agree on first coordinates, for each $i\in R$ we can compare
the valuation $\mu_i$ of the corresponding image under $h_1$ with the
valuation $\nu_i$ of its image under $p$.  For distinct $i,j\in R$,
the two morphisms give the same adjacency between these images, because
they both preserve or both complement $E$.  By \eqref{eq:looped-adjacency},
\[
 \mu_i(j)+\nu_i(j)=\mu_j(i)+\nu_j(i).
\]
Thus, if both $i,j\in R$, the required changes to these two entries
agree, exactly as needed for $\gamma_{ij}$.  If only one of $i,j$ belongs
to $R$, the change in the other coordinate does not matter.  Apply
$\gamma_{ij}$ for each pair where a change is needed.  Since these maps
commute, their product corrects every specified valuation without changing
the first two coordinates.  Composing it with $h_1$ extends $p$.
\end{proof}

The complementing graph reduct thus has two natural incomparable expansions with EPPA:
ordinary graphs, and graphs with loops modulo simultaneous complementation of
edges and loops.  Viewed as first-order reducts of the random graph with a generic unary
predicate, their meet is the complementing graph reduct itself, which does
not have EPPA.  This also answers
the question about this meet raised in~\cite[Section~4.1]{BradleyEPPAnumbers}.\footnote{%
Let $(R,P)$ be the random graph with a generic unary predicate; view $P$
as the loop relation.  By uniqueness of the Fra\"{\i}ss\'e limit, there is a
permutation $\delta$ simultaneously complementing adjacency and $P$.
The automorphism groups of the two expansions are $\Aut(R)$ and
$\Aut(R,P)\cup\delta\Aut(R,P)$, respectively.  Since
$\Aut(R,P)\le\Aut(R)$, they generate
$\Aut(R)\cup\delta\Aut(R)$, precisely the automorphism group of
the graph modulo global complementation.}

\medskip

One might wonder whether Beliayeu's positive result is explained simply by
the absence of complementing partial morphisms with fixed points: these are
precisely what our obstruction in Proposition~\ref{prop:a6} uses.  But
forbidding fixed points is not enough.  In fact, $\cG^{\pm}$ fails even the
weaker extension property where we only consider partial automorphisms with
disjoint domain and range.  To see this, suppose that every structure in
$\cG^{\pm}$ had a finite extension in which all such partial automorphisms
extend.
Given $A\in\cG^{\pm}$, take two disjoint isomorphic copies $A_0,A_1$
of $A$ inside a structure $\widetilde A\in\cG^{\pm}$, and let
$\iota\colon A_0\to A_1$ be the natural isomorphism.  For any partial
automorphism $\varphi$ of $A_0$, both $\iota^{-1}$ and
$\iota\circ\varphi$ have disjoint domains and ranges.  Extending
both in a common witness for $\widetilde A$ and composing the extensions
would extend $\varphi$.  This would give EPPA for $\cG^{\pm}$,
contradicting Theorem~\ref{thm:complementing-graph}.

\subsection{Equivalence relations with finitely many classes}

Ivanov proved EPPA for equivalence relations on $r$-element subsets with no restriction on the number of equivalence classes~\cite{Ivanov2015}.  The negative argument of
Section~\ref{sec:complementing-graphs} extends directly to disprove EPPA
 when the number of classes is bounded.

\begin{proposition}\label{prop:bounded-equivalence}
For every $r,m\ge2$, the class of finite sets equipped with an equivalence
relation on $r$-element subsets having at most $m$ equivalence classes does not
have EPPA.  The same holds for equivalence relations on ordered $r$-tuples of
distinct elements.
\end{proposition}

\begin{proof}
We give the argument for $r$-element subsets.  For a finite set $V$,
we regard an equivalence relation on $\binom Vr$ as a structure with
underlying set $V$, not $\binom Vr$: a partial automorphism is a
partial bijection of $V$ whose induced action on $r$-subsets preserves
the equivalence relation.  As in Section~\ref{sec:complementing-graphs},
we represent an equivalence relation on $\binom Vr$ by a function
$f\colon\binom Vr\to\mathbb Z_m$: two subsets are equivalent precisely when
they have the same value.  We call the values colours; partial automorphisms
may permute them, since only equality and inequality of values matter.
We generalise Proposition~\ref{prop:a6} by constructing a cyclic
permutation $p$ of a set $S$ which also cyclically permutes the $m$ colours.
We then add $r$ vertices, each individually fixed by a partial automorphism
extending $p$, to apply Lemma~\ref{lem:coset} with $b=r-1$.
For $r=m=2$, suitable choices give precisely $A_6$.

Put $M=\operatorname{lcm}(1,\ldots,r)$, $N=mM$, $S=\mathbb Z_N$, and
let $p\colon S\to S$ be the translation $p(x)=x+1$.  We also write
$pY=\{p(y):y\in Y\}$ for its induced action on subsets.  We want to colour
the $r$- and $(r-1)$-subsets of $S$ so that $p$ increases their colours
by $1$ modulo $m$.  Let $Y\subseteq S$ have size $t\in\{r-1,r\}$, and
let $q$ be the size of its orbit under the induced action of $p$.  Then
$q\mid N$ and $p^qY=Y$.  Since $p$ is an $N$-cycle on $S$, its power
$p^q$ consists of cycles of length $\frac{N}{q}$ on $S$, and hence $Y$
is a union of such cycles, so $\frac{N}{q}\mid t$.  Consequently $\frac{N}{q}\mid M$
and hence $m\mid q$.  We can therefore colour each translation orbit
by repeating $0,1,\ldots,m-1$, obtaining colourings
$c\colon\binom Sr\to\mathbb Z_m$ and
$\alpha\colon\binom S{r-1}\to\mathbb Z_m$ which both increase by $1$ under $p$.
Choose $r$ pairwise distinct such colourings of $(r-1)$-subsets,
$\alpha_1,\ldots,\alpha_r$.  For $r=2$ take $\alpha,\alpha+1$.
For $r\ge3$, there are at least $r-1$ translation orbits on $\binom S{r-1}$.
Indeed, since $N\ge 2r$, for every $j=0,\ldots,r-2$ we can choose an
$(r-1)$-subset containing exactly $j$ pairs of consecutive elements of
$\mathbb Z_N$.  Translation preserves this number, so these subsets belong
to distinct orbits.  Choose $r-1$ such orbits $O_1,\ldots,O_{r-1}$ and take
$\alpha$ and $\alpha+\mathbf1_{O_i}$.

Let $V=S\cup\{a_1,\ldots,a_r\}$, where the $a_i$ are new vertices, and
define $f\colon\binom Vr\to\mathbb Z_m$ by $f(Y)=c(Y)$ for
$Y\in\binom Sr$ and $f(\{a_i\}\cup Y)=\alpha_i(Y)$ for every $i$ and
every $Y\in\binom S{r-1}$.  Assign arbitrary values to the remaining
$r$-subsets, and let $A$ be the resulting equivalence structure.

The function $c$ uses all $m$ values.  For each $i$, put
$p_i=p\cup\{a_i\mapsto a_i\}$, a bijection of $S\cup\{a_i\}$.
Every $r$-subset $Y$ of this domain satisfies
$f(p_iY)=f(Y)+1$, so $p_i$ is a partial automorphism of $A$.
Suppose that $B$ is an EPPA-witness and choose automorphisms
$g_i\in\Aut(B)$ extending $p_i$.  Put $L=\Aut(B)_{(S)}$, the pointwise
stabiliser of $S$ in the action on vertices.  Since $S$ realises all
$m$ classes, each element of $L$ fixes every class.  As $g_i|_S=p$
for every $i$, the permutation $g_i g_1^{-1}$ fixes $S$ pointwise and
thus belongs to $L$.  Hence all $g_i\in Lg_1$, and $g_1$ normalises $L$
because $g_1(S)=S$.  Every element of $Lg_1$ permutes the $m$ classes
by adding $1$ to their colours, so it fixes at most $r-1$ vertices.
By Lemma~\ref{lem:coset} with $b=r-1$, fixed points of elements of
$Lg_1$ lie in at most $r-1$ $L$-orbits.  Since $g_i$ fixes $a_i$,
two of the $a_i$ are $L$-conjugate.  If $\ell(a_i)=a_j$ for some
$\ell\in L$, then
$f(\{a_i\}\cup Y)=f(\{a_j\}\cup Y)$ for every
$Y\in\binom S{r-1}$, because $\ell$ fixes $S$ pointwise and preserves
each equivalence class.  Hence $\alpha_i=\alpha_j$, a contradiction.

For ordered injective tuples take instead $S=\mathbb Z_{mr}$.  Translation
acts freely on ordered tuples of lengths $r$ and $r-1$, so choose
colourings $c$ and $\alpha$ which increase by $1$ under translation.
For an external vertex $a_i$, colour a tuple with $a_i$ in position $j$
by $\alpha$ of the remaining ordered tuple plus $\mathbf1_{i=j}$
modulo $m$.  These $r$ types are distinct, since the types for $a_i$ and $a_j$ differ
on tuples with the external vertex in position $i$ whenever $i\ne j$.
Moreover, each $p\cup\{a_i\mapsto a_i\}$ increases all defined colours by
$1$, so the same coset argument applies.
\end{proof}

\subsection{Reducts of random hypergraphs}

Taking $m=2$ in Proposition~\ref{prop:bounded-equivalence} shows that finite
$r$-uniform hypergraphs modulo global complementation do not have EPPA for
every $r\ge2$.  This is one of the negative cases in the higher-arity analogue
of Theorem~\ref{thm:classification}.  Thomas classified the reducts of the
countable random $r$-uniform hypergraph~\cite{thomas1996reducts}, and in
forthcoming work we give their complete EPPA classification.  For every
$r\ge2$, the only positive cases are the pure set, the random $r$-uniform
hypergraph itself, and one natural switching reduct, which can be
represented by $(r+1)$-uniform hypergraphs in which every $(r+2)$-element set
contains an even number of hyperedges; for $r=2$ this is precisely the class
of two-graphs.  Thus Theorem~\ref{thm:classification} is the first non-trivial
instance of the corresponding higher-arity picture.

\bigskip

\paragraph{\textbf{Formal verification.}} Theorem~\ref{thm:classification} and Proposition~\ref{prop:bounded-equivalence} have been formalised in Lean~4; see \url{https://github.com/shulda/lean_eppa_reducts}. The Lean formalisation was developed almost entirely autonomously by AI, primarily as an additional check on our proofs, without aiming to meet the usual standards of style and reusability in formalised mathematics.

\medskip

\paragraph{\textbf{AI disclaimer.}}
This work was developed with iterative use of ChatGPT (GPT-5.6 Sol and GPT-6),
both in the mathematical research and in drafting and editing the manuscript.
The authors have directed, reviewed and verified the work in full and take responsibility
for all its contents.

\bibliographystyle{alpha}
\bibliography{ramsey}

\end{document}